\documentclass[11pt,a4paper]{article}
\usepackage{authblk}
\usepackage{graphicx}
\usepackage{amsmath}
\usepackage{amssymb}
\usepackage{algorithm}
\usepackage{algorithmic}
\usepackage{amsthm}
\usepackage{epstopdf}
\usepackage{caption}
\usepackage{url}
\usepackage{subcaption}
\usepackage{natbib}
\usepackage{color}
\usepackage{cases}
\newcommand\inner[2]{\langle #1, #2 \rangle}
\newtheorem{theorem}{Theorem}
\newtheorem{lemma}{Lemma}

\newtheorem{corollary}{Corollary}

\theoremstyle{definition}

\title{\bf Non-convex Conditional Gradient Sliding
 }
\date{}
\author[1]{Chao Qu}
\author[2]{Yan Li}
\author[2]{Huan Xu}
\affil[1]{Department of Mechanical Engineering, National University of Singapore}
\affil[2]{H. Milton Stewart School of Industrial and Systems Engineering, Georgia Institute of Technology}

\begin{document} 
	\maketitle

\begin{abstract}
	We investigate a projection free method, namely conditional gradient sliding  on  batched, stochastic and finite-sum non-convex problem.  CGS is a smart combination of Nesterov's accelerated gradient method and Frank-Wolfe  (FW) method, and outperforms  FW  in the convex setting by saving gradient computations. However, the study of CGS in the non-convex setting is limited. In this paper, we propose the non-convex conditional gradient sliding (NCGS) which surpasses the non-convex Frank-Wolfe method in batched, stochastic  and finite-sum setting. 
\end{abstract}

\section{Introduction}

We study the following problem
\begin{equation}
\min_{\theta \in \Omega} F(\theta), 
\end{equation}
where $F(\theta)$ is non-convex and $L$ smooth, $\Omega$ is a complex constraint. In the stochastic setting, we assume $F(\theta)=E_{\xi} f(\theta,\xi)$, where $f(\theta,\xi)$ is smooth and non-convex, while in the finite-sum case, we have $F(\theta)=\frac{1}{n}\sum_{i=1}^n f_i(\theta).$

In many real setting, the cost of projection on $\Omega$  may be expensive~(for instance, the projection on the trace norm ball) or even computationally intractable \citep{collins2008exponentiated}. To alleviate such difficulty, the Frank-Wolfe method \citep{frank1956algorithm} (a.k.a. Conditional Gradient method ) which was initially developed for the convex problem in 1950s attracts attentions again in machine learning community \citep{jaggi2013revisiting}, due to its projection free property. In each iteration, the algorithm calls first-order oracle to get $\nabla F(\theta)$ and then calls a linear oracle in the form $\arg\min_{\theta\in \Omega}\langle \theta, g \rangle $, which avoids the projection operation. It is well-known that, to achieve $\epsilon$-solution, $\mathcal{O}\big(\frac{1}{\epsilon}\big)$ iterations are required given that $F(\theta)$ is convex and smooth. However this rate is significantly worse than the optimal rate $\mathcal{O} \big( 1/\sqrt{\epsilon}  \big)$ for the smooth convex problem \citep{nesterov2013introductory}, which raises a question whether this complexity bound $\mathcal{O} ( \frac{1}{\epsilon})$ is improvable. Unfortunately, the answer is no in the general setting \citep{lan2013complexity,guzman2015lower} and the improvable result can only be obtained with stronger assumptions, see works in \citep{garber2013linearly,garber2015faster}. In \citep{lan2016conditional}, the author proposes the conditional gradient sliding method which combines the idea of Nesterov's accelerated gradient with Frank-Wolfe method. While the number of calls on linear oracle is same, the number of gradient computations (the first order oracle) is significantly improved from $ \mathcal{O} (\frac{1}{\epsilon}) $ to $\mathcal{O}(\frac{1}{\sqrt{\epsilon}})$. Under the strongly convex assumption, this bound can be further pushed into $  \mathcal{O}( \log(1/\epsilon))$ by using the restarting skills \citep{lan2016conditional}.

Very recently, people have investigated the convergence on non-convex Frank-Wolfe method, which includes the batched, stochastic and finite-sum setting \citep{lacoste2016convergence,reddi2016stochasticfw}. A natural question is that does the same thing happens when we combine the Nestrov's accelerated gradient
 with Frank-Wolfe method in the non-convex setting . Our answer is yes, the non-convex conditional gradient sliding (NCGS) improves the complexity on the first order oracle. \textbf{Our contributions}  are summarized in the table \ref{table:batch},\ref{table:stochastic}, \ref{table:finite-sum} (with red color). To the best of our knowledge, our result outperforms  Frank-Wolfe method in the corresponding setting. For the formal definition of first order oracle (FO), stochastic first order oracle (SFO), Incremental First Order Oracle (IFO) and linear oracle (LO), see section \ref{section.preliminary}.
 

\begin{table}[h]
		\centering
	\begin{tabular}{| l |l |l|}
		\hline
	
			Algorithm  & FO complexity & LO complexity\\ \hline
			\color{red}{NCGS  }      & \color{red}{ $\mathcal{O} (1/\epsilon) $  }    &  \color{red} {$\mathcal{O} (1/\epsilon^2)$}  \\ \hline
			FW        &   $\mathcal{O} (1/\epsilon^2)$        &  $\mathcal{O} (1/\epsilon^2)$   \\ \hline
	\end{tabular}
	\caption{Comparison of complexity of algorithms in the batched setting.}\label{table:batch}
\end{table}

\begin{table}[h]
	\centering
	\begin{tabular}{| l |l |l|}
		\hline
			Algorithm  & SFO complexity & LO complexity\\ \hline
		\color{red} {NCGS}        &      \color{red} {$\mathcal{O} (1/\epsilon^2)$}       &   \color{red}{$\mathcal{O} (1/\epsilon^2)$ }  \\ \hline
	SAGAFW      &  $\mathcal{O} (1/\epsilon^{\frac{10}{3}})$         & $\mathcal{O} (1/\epsilon^2)$    \\ \hline
	SVFW         &   $\mathcal{O} (1/\epsilon^{\frac{8}{3}})$        &  $\mathcal{O} (1/\epsilon^2)$   \\  \hline
	\end{tabular}
	\caption{Comparison of complexity of algorithms in the stochastic setting.}\label{table:stochastic}	
\end{table}

\begin{table}[h]
	\centering
	\begin{tabular}{| l |l |l|}
		\hline
		Algorithm  & IFO complexity & LO complexity\\ \hline
		\color{red}{NCGS  }      & \color{red}{ $\mathcal{O} (n/\epsilon) $  }    &  \color{red} {$\mathcal{O} (1/\epsilon^2)$}  \\ \hline
		FW          &   $\mathcal{O} (n/\epsilon^2)$        &  $\mathcal{O} (1/\epsilon^2)$   \\ \hline
		\color{red} {NCGS-VR }        &   \color{red}{$\mathcal{O} (\frac{n^{\frac{2}{3}}}{\epsilon})$ }        &  \color{red} {$\mathcal{O} (1/\epsilon^2)$} \\ \hline
	SVFW        & $\mathcal{O}(n+\frac{n^{\frac{2}{3}}}{\epsilon^2})$        &  $\mathcal{O} (1/\epsilon^2)$ \\ \hline 		
	\end{tabular}
	\caption{Comparison of complexity of algorithms in the finite-sum setting. Since we need to evaluate n gradients each iteration in NCGS and FW, the IFO complexity of NCGS and FW are  n$\times$  results in table \ref{table:batch}.}\label{table:finite-sum}	
\end{table}
 
\subsection*{Related work}

The classical Frank-Wolfe method considers the smooth convex function $F(\theta)$ over a polyhedral constraint and enjoys $\mathcal{O} (1/\epsilon)$ convergence rate \citep
{frank1956algorithm,jaggi2013revisiting}. Recent work in \citep{garber2013linearly,garber2015faster} proves faster convergence rate given additional assumption. The conditional gradient sliding proposed in \cite{lan2016conditional} aims at the convex objective function. While our high level idea is same with it, the analysis is totally different due to the non-convexity.

Most non-convex work includes the projection or proximal operation, and we just list some of them below. The authors in \citep{ghadimi2013stochastic} investigate SGD in the non-convex setting. They extend Nesterov's acceleration method in the constrained stochastic optimization \citep{ghadimi2016accelerated}.  The performance on non-convex stochastic variance reduction method is analyzed in \citep{reddi2016stochastic,allen2016variance,shalev2016sdca,allen2016improved}. 

The literature on projection free method in non-convex optimization is very limited. The early work in \citep{bertsekas1999nonlinear} proves the asymptotic convergence of Frank-Wolfe method to the stationary point but without rate. In \citep{lacoste2016convergence}, the author provides the rate $\mathcal{O} (1/\epsilon^2)$ for the Frank-Wolfe method in (batched) non-convex setting with the criteria of Frank-Wolfe gap, i.e.,  both FO and LO complexity are $\mathcal{O} (1/\epsilon^2)$,  while in our NCGS, $FO$ complexity is $\mathcal{O} (1/\epsilon)$ and LO complexity is $ \mathcal{O}(1/\epsilon^2)$ .     Recent work on stochastic Frank-Wolfe method (non-convex) shows that SFO complexity and LO complexity are $\mathcal{O} (1/\epsilon^{\frac{10}{3}})$, $ \mathcal{O} ( 1/\epsilon^2) 
$ respectively in SVFW-S and $  \mathcal{O} (1/\epsilon^{\frac{8}{3}})$, $ \mathcal{O} ( 1/\epsilon^2)$ in SAGAFW-S \citep{reddi2016stochasticfw}. Our SFO and LO on the same setting are $\mathcal{O}(1/\epsilon^2)$ and  $\mathcal{O}(1/\epsilon^2)$ respectively. In the finite sum setting, our variance reduction NCGS(NCGS-VR) has IFO complexity $ \mathcal{O} (\frac{n^{\frac{2}{3}}}{\epsilon}) $,while the state of the art variance reduced FW has complexity $\mathcal{O}(n+\frac{n^{\frac{2}{3}}}{\epsilon^2})$ and the same LO complexity \citep{reddi2016stochasticfw}.  It is clear that no matter in batched, stochastic and finite sum setting, our results outperforms the literature by saving  gradient computations.
\section{Preliminary }\label{section.preliminary}

\subsection*{Oracle model}
\begin{itemize}
	\item First Order Oracle (FO): FO returns $\nabla_{\theta} F(\theta)$.
	\item Stochastic First Order Oracle (SFO): For a function $F(\theta)=\mathbb{E}_{\xi} f(\theta,\xi)$ where $\xi\sim P$, a SFO returns the stochastic gradient $G(\theta_k,\xi_k)=\nabla_{\theta}f(\theta_k,\xi_k)$ where $\xi_k$ is a sample drawn i.i.d. from $P$ in the $k$ th call.
	\item Incremental First Order Oracle (IFO): For the setting $F(\theta)=\frac{1}{n}\sum_{i=1}^{n}f_i(\theta)$, an IFO samples $i\in [n]$ and returns $ \nabla_\theta f_i(\theta).$
	\item Linear oracle (LO): LO solves the following problem $\arg\min_{\theta\in\Omega} \langle \theta, g   \rangle  $ for a given vector $g$.
\end{itemize}

Thought out the paper, the complexity of $FO$, $SFO$, $IFO$, $LO$ denotes the number of call of them to obtain solution with  $\epsilon$ accuracy. 
 
\subsection*{Assumptions}
We say $F(\theta)$ is $L$ smooth, if $\|\nabla F(\theta_1)-\nabla F(\theta_2)\|\leq L \|\theta_1-\theta_2\|$. This definition is equivalent to the following form $-\frac{L}{2}\|\theta_1-\theta_2\|^2\leq F(\theta_1)-F(\theta_2)-\langle F(\theta_2),\theta_1-\theta_2  \rangle\leq \frac{L}{2}\|\theta_1-\theta_2\|^2, \forall \theta_1, \theta_2 \in \Omega$

We say $F(\theta)$ is $\ell$ lower smooth if it satisfies 
$$    -\frac{l}{2}\|\theta_1-\theta_2\|^2\leq F(\theta_1)-F(\theta_2)-\langle F(\theta_2),\theta_1-\theta_2  \rangle, \forall \theta_1,\theta_2 \in \Omega.$$

Easy to see that if we just have L smooth assumption, then $l=L$. However, in some cases, the non-convexity $l$ is much smaller than $L$ and we will show how it affects the result in our theorem.

We then define some prox-mapping type function $\psi(x,\omega,\gamma)$:
$$ \psi(x,\omega,\gamma)=\arg\min_{\theta\in \Omega} \langle \omega,\theta\rangle+\frac{1}{2\gamma}\|\theta-x\|^2.$$
It is closely related to the projected gradient by setting $w=\nabla F(\theta
)$, $\gamma$ by the stepsize and $x=\theta_{k}$.  We assume $\psi(x,\omega,\gamma)\leq M $ for all $\gamma \in (0,\infty)$ and $ x\in \Omega $ and $\omega\in R^p$.

For the stochastic setting, we have following additional assumptions.

For any  $\theta\in \mathbb{R}^p$ and $k>1$, we have
$$ 1. \mathbb{E} G(\theta,\xi_k)=\nabla F(\theta)~~2. \mathbb{E} \|G(\theta,\xi)-\nabla F(\theta)\|^2\leq \sigma^2, $$
which are unbiasedness and bounded variance assumption on $G(\theta,\xi_k)$ respectively.

\subsection*{Convergence criteria}
The conventional convergence criteria in non-convex optimization to find a solution with $\epsilon$ accuracy is $\|\nabla F(\theta)\|^2\leq \epsilon$ \citep{lan2016conditional,nesterov2013introductory}. However when the problem has constraint, it needs a different termination criterion based on the gradient mapping \citep{lan2016conditional}, which is a natural extension of gradient (if there is no constraint, it reduces to the gradient.)

We define the gradient mapping as follows

$$g(\theta,\nabla F(\theta),\gamma)=\frac{1}{\gamma} (\theta-\psi(\theta,\nabla F(\theta),\gamma)).$$

Through out the paper, we use $g(\theta,\nabla F(\theta),\gamma)$ as the convergence criteria, i.e., we want to find the solution $\theta$ such that $\|g(\theta,\nabla F(\theta),\gamma)\|^2\leq \epsilon$. Notice there is another criteria called Frank-Wolfe gap $\max_{x\in \Omega} \langle x-\theta_k, -\nabla F(\theta_k) \rangle $ in the analysis of recent non-convex Frank-Wolfe method \citep{lacoste2016convergence,reddi2016stochastic}, which was initially used in the convex Frank-Wolfe method. While, in this paper, we follow the definition on gradient mapping, since it is a natural generalization of gradient.

\section{Batched Non-convex conditional gradient sliding}
\subsection{Algorithm}

\begin{algorithm}[h]
	\caption{Non-convex conditional gradient sliding (NCGS)}
	\label{alg:non-convex-CGS}
	\begin{algorithmic}
		\STATE {\bfseries Input:}  Step size $\alpha_k$, $\lambda_k$, $\beta_k$, smoothness parameter $L$.
		\STATE Initialization: $\theta_0^{ag}=\theta_0$, k=1.\\
		\FOR{$k=1,...,N$}
		\STATE update:  $\theta_k^{md}=(1-\alpha_k)\theta_{k-1}^{ag}+\alpha_{k}\theta_{k-1}$\\
		\STATE update: $\theta_k=condg (\nabla F(\theta_{k}^{md}),\theta_{k-1}, \lambda_k,\eta_k) $\\
		\STATE update: option I: $\theta_{k}^{ag}=\theta_k^{md}- \beta_k \tilde{g}(\theta_{k-1},\nabla F(\theta_{k}^{md}),\beta_k,\eta_k) $, where $\tilde{g}(\theta_{k-1},\nabla F(\theta_{k}^{md}),\beta_k,\eta_k):=\frac{\theta_{k-1}-\theta_k}{\lambda_k} .$\\
		\qquad \quad option II: $\theta_k^{ag}=condg(\nabla F(\theta_k^{md}),\theta_k^{md},\beta_k, \chi_k).$
		\ENDFOR
	\end{algorithmic}
\end{algorithm}
The procedure of $condg$ is presented below.
\begin{algorithm}[h]
	\begin{algorithmic}
		\caption{Procedure of $ u^{+}=condg (l, u,\lambda,\eta)$}
		\STATE 1.$u_1=u$ and $t=1$.
		\STATE 2.$v_t$ be an optimal solution for the subproblem $$ V(u_t)=\max_{x\in \Omega}\langle l+\frac{1}{\gamma}(u_t-u), u_t-x  \rangle $$
		\STATE 3.if $V(u_t)\leq \eta$, set $u^{+}=u_t$ and terminate the procedure
		\STATE 4.$u_{t+1}=(1-\xi_t)u_t+\xi_tv_t$ with $\xi_t=\min \{ 1, \frac{  \langle \frac{1}{\gamma}(u-u_t)-l,v_t-u_t \rangle }{\frac{1}{\gamma} \|v_t-u_t\|^2}  \} $
		\STATE Set $t\leftarrow t+1$ and go to step 2.
		\STATE \textbf{end procedure} 
	\end{algorithmic}
\end{algorithm}

\subsection{Theoretical result}
\begin{theorem}\label{Theorem.batch}
	Set $\alpha_k=\frac{2}{k+1}$, $\beta_k=\frac{1}{2L}$, $\lambda_k =\beta_k$, $\eta_k=\frac{1}{N}$ in option $I$ in Algorithm \ref{alg:non-convex-CGS}, then we have 
	
	$$\min_{k=1,...,N} \|g(\theta_{k-1},\nabla F(\theta_{k}^{md}),\lambda_k)\|^2\leq \frac{12L(F(\theta_0)-F(\theta^*))+16L}{N}. $$  In option II,  we set $\alpha_k=\frac{2}{k+1}$, $\beta_k=\frac{1}{2L}$,$\lambda_k=k\beta_k/2$,  $\eta_k=\frac{1}{N}$,$\chi_k=\frac{1}{N}$, then we have
	
	$$ \min_{k=1,...,N} \|g(\theta_{k-1},\nabla F(\theta_{k}^{md}),\beta_k)\|^2\leq  48L \big(  \frac{4L \|\theta_0-\theta^*\|^2}{N^2(N+1)}+\frac{l}{N} (\|\theta^*\|^2+2M^2)+\frac{2}{N}  \big).$$ 
	
\end{theorem}

Remarks:  The FO complexities of option I and II are same ,i.e., $\mathcal{O} (1/\epsilon)$.  However, when the non-convexity $l$ is small, option II has better convergence rate. In the high level,  $\frac{L^2 \|\theta_0-\theta^*\|^2}{N^2(N+1)}$ corresponds to the convex part of the function,  $\frac{Ll}{N} (\|\theta^*\|^2+2M^2)$ corresponds to the non-convex part of the function, while the last term $L/N$ corresponds to the procedure of condg.

Using above theorem, we have following corollary on FO and LO complexity.

\begin{corollary}\label{Corollary.batch}
	Under the same condition of theorem \ref{Theorem.batch}. In option I and II of algorithm \ref{alg:non-convex-CGS}, to achieve the accuracy $\epsilon$, the FO complexity is $ \mathcal{O}(1/\epsilon) $ and the LO complexity is $ \mathcal{O}(\frac{1}{\epsilon^2}) $.	
\end{corollary}

\section{Stochastic non-convex conditional gradient sliding}

In this section we consider the following problem 
\begin{equation}
\min_{\theta \in \Omega} F(\theta):=E_{\xi} f(\theta,\xi). 
\end{equation}

\subsection{Algorithm}

The stochastic CGS method is obtained by replacing the exact gradient $\nabla F(\theta)$ in Algorithm \ref{alg:non-convex-CGS} by the stochastic gradient $G(\theta,\xi)$, incorporating a mini-batched approach, and randomized termination criterion for the non-convex stochastic optimization in \citep{ghadimi2016accelerated}. In particular, we define $\bar{G}_k=\frac{1}{m_k}\sum_{i=1}^{m_k} G(x_k^{md},\xi_{k,i}).$

Notice, using our assumption in section \ref{section.preliminary}, we have

$$ \mathbb{E} \bar{G}_k=\frac{1}{m_k}\sum_{i=1}^{m_k} \mathbb{E} G(\theta_k^{md},\xi_{k,i})=\nabla F(\theta_k^{md}),$$
and 
\begin{equation}\label{equ.average_variance}
\mathbb{E} \|\bar{G}_k-\nabla F(\theta_k^{md})\|^2\leq \frac{\sigma^2}{m_k}.
\end{equation}

\begin{algorithm}[h]
	\caption{Stochastic Non-convex conditional gradient sliding}
	\label{alg:non-convex-SCGS}
	\begin{algorithmic}
		\STATE {\bfseries Input:}  Step size $\alpha_k$, $\lambda_k$, $\beta_k$, smoothness parameter $L$, a probability mass function $P_R(\cdot)$ with $Prob\{R=k \}=p_k, k=1,...,N$. \STATE Initialization: $\theta_0^{ag}=\theta_0$, k=1.\\
		\STATE Let $R$ be a random variable 
		\FOR{$k=1,...,R$}
		\STATE update:  $\theta_k^{md}=(1-\alpha_k)\theta_{k-1}^{ag}+\alpha_{k}\theta_{k-1}$\\
		\STATE update: $\theta_k=condg ( \bar{G}_k,\theta_{k-1},\lambda_k,\eta_k) $\\
		\STATE update:
		$\theta_k^{ag}=condg( \bar{G}_k,\theta_k^{md},\beta_k, \chi_k)$
		\ENDFOR
	\end{algorithmic}
\end{algorithm}

\subsection{Theoretical Result}
\begin{theorem}\label{Theorem.stochastic}
	Set $\alpha_k=\frac{2}{k+1}$, $\beta_k=\frac{1}{2L}$, $\lambda_k=\frac{k\beta_k}{2}$, $\eta_k=\chi_k=\frac{1}{N}$ and $m_k=k$ , and set $p_k=\frac{\Gamma_k^{-1}}{\sum_{k=1}^{N}\Gamma_k^{-1}}$, where $\Gamma_k=\frac{2}{k(k+1)}$  in Algorithm \ref{alg:non-convex-SCGS}, then we have

	$$ E [ \|g(\theta_{R}^{md},\bar{G}_R,\beta_R)\|^2]\leq 192L \big(  \frac{4L \|\theta_0-\theta^*\|^2}{N^2(N+1)}+\frac{l}{N} (\|\theta^*\|^2+2M^2)+\frac{1}{N}+\frac{3\sigma^2}{2LN}   \big). $$
\end{theorem}

Compare this result with its batched counterpart, i.e., theorem \ref{Theorem.batch}, we see there is a additional term $\frac{\sigma^2}{N}$ in the upper bound  corresponding to the variance of the gradient.

Using this theorem, we obtain the LO and SFO complexity in the following corollary.

\begin{corollary}\label{Corollary.stochastic}
	Under the same setting of Theorem \ref{Theorem.stochastic},  SFO and LO complexities in algorithm \ref{alg:non-convex-SCGS} are $\mathcal{O} (1/\epsilon^2)$ and $\mathcal{O} (1/\epsilon^2) $ respectively.
\end{corollary} 

Notice in algorithm \ref{alg:non-convex-SCGS}, we use the mini-batch to calculate $\bar{G}_k$.Thus even the total iteration of the stochastic non-convex conditional gradient sliding is same with the batched one, it needs more calls of SFO.

\section{Finite-Sum nonconvex conditional gradient sliding}
\subsection{Algorithm}
In this section we consider minimizing finite sum problem
\begin{equation}\label{finite_sum}
\min_{x \in \mathbb{R}^d} F(x)=\frac{1}{n}\sum_{i=1}^n f_i(x),
\end{equation}
where each $f_i$ is possibly nonconvex, but  smooth with parameter $L$. If we view the finite-sum problem as a special case of batch
problem, then use our algorithm \ref{alg:non-convex-CGS}, we have IFO complexity $\mathcal{O}(\frac{n}{\epsilon})$. Variance reduction technique has been proposed for finite sum problem to reduce dependence of IFO
complexity on number of component $n$. We incorporate a popular technique, namely SVRG, into our previous algorithm. We will show our new algorithm achieve IFO complexity $\mathcal{O}(\frac{n^{\frac{2}{3}}}{\epsilon})$, 
and with proper specification of parameter, the algorithm can in fact achieve IFO complexity of $\mathcal{O}(\min\{\frac{1}{\epsilon^2},\frac{n^{\frac{2}{3}}}{\epsilon}\})$. To the best of our knowledge, our algorithm outperforms any Frank-Wolfe type algorithm for the non-convex finite-sum problem.

\begin{algorithm}[h]
	\caption{Variance reduction Non-convex conditional gradient sliding (NCGS-VR)}
	\label{alg:non-convex-finite-sum-SCGS}
	\begin{algorithmic}
		\STATE {\bfseries Input:}  $\tilde{\theta}_0=\theta_m^0=\theta_0 \in \mathbb{R}^d$, epoch length $m$, stepsize $\lambda$, proximal tolerance $\eta$, minibatch size $b$, iteration limit $T$, $S=\frac{T}{m}$.
		\FOR{$s=0,...,S-1$}
		\STATE $\theta_0^{s+1}=\theta_m^s$\\
		\STATE $g^{s+1}=\frac{1}{n} \sum_{i=1}^n \nabla f_i(\tilde{\theta}^s)$\\
		\FOR{$t=0,\ldots,m-1$}
		\STATE Pick $I_t$ uniformly from $\{1,\ldots,n\}$ with replacement such that $|I_t|=b$\\
		\STATE $v_t^{s+1}=\frac{1}{b}\sum_{i \in I_t}(\nabla f_{i_t}(\theta_t^{s+1})-\nabla f_{i_t}(\tilde{\theta}^s))+g^{s+1}$\\
		\STATE $\theta_{t+1}^{s+1}=cndg(v_t^{s+1},\theta_t^{s+1},\lambda_t,\eta_t)$\\
		\ENDFOR\\
		\STATE $\tilde{\theta}^{s+1}=\theta_{m}^{s+1}$
		\ENDFOR\\
		\STATE {\bfseries Output:} Iterate $\theta_{\alpha}$ chosen uniformly at random from $\{\{\theta_t^{s+1}\}_{t=0}^{m-1}\}_{s=0}^{S-1}$.
	\end{algorithmic}
\end{algorithm}

\subsection{Theoretical result}
\begin{theorem}\label{finite_sum_NCGS}
	Let $b=n^{\frac{2}{3}}$ in algorithm \ref{alg:non-convex-finite-sum-SCGS}, let $\lambda_t=\frac{1}{3L}$, $m=n^{\frac{1}{3}}$ and $T$ is a multiple of m, let $\eta=\frac{1}{T}$. Then for output $\theta_a$ we have:
	$$ \mathbb{E}[\| g(\theta_{\alpha},\nabla F(\theta_{\alpha}),\lambda)\|^2] \leqslant \frac{18L(f(\theta_0)-f(\theta^{\star})+1)}{T}$$
	where $\theta^{\star}$ is an optimal solution to (\ref{finite_sum}).
\end{theorem}

Using result from previous theorem, we obtain the following IFO and LO complxity.

\begin{corollary}
	Let parameters be set as in theorem \ref{finite_sum_NCGS}, the IFO and LO complexities of Algorithm \ref{alg:non-convex-finite-sum-SCGS} are $\mathcal{O}(\frac{n^{\frac{2}{3}}}{\epsilon})$ and $\mathcal{O}(\frac{1}{\epsilon^2})$ respectively for
	$ \mathbb{E}[\| g(\theta_{\alpha},\nabla F(\theta_{\alpha}),\lambda)\|^2] \leqslant \epsilon$.
\end{corollary}

\section{Simulation Result}
We consider a modified  matrix completion  problem for our simulation. In particular, we optimize the following trace norm constrained non-convex problem with candidate algorithms.

\begin{equation}\label{equ.matrix_completion}
\min_{\theta} \sum_{(i,j)\in \Omega} f_{i,j} (\theta) \quad \mbox{s.t.} \quad\|\theta\|_{*}\leq R,
\end{equation}
where $\Omega$ is the set of observed entries, $f_{i,j}= \big( 1-\exp(-\frac{(\theta_{i,j}-Y_{i,j})^2}{\sigma})\big)$, $Y_{i,j}$ is the observation of $(i,j)$'s entry, $\|\cdot\|_{*}$ is the nuclear norm. Here $f_{i,j}$ is a smoothed $\ell_0$ loss with enhanced robustness to outliers in the data, thus it can solve sparse+low rank matrix completion in \citep{chandrasekaran2009sparse}. Obviously, this $f_{i,j}$ is non-convex and satisfies assumptions in our algorithm \ref{alg:non-convex-CGS},\ref{alg:non-convex-SCGS},\ref{alg:non-convex-finite-sum-SCGS}.

We compare our non-convex conditional gradient sliding method with Frank-Wolfe method in Fig \ref{fig:matrix_completion}. Particularly, we report the result of the batched setting in Figure\ref{fig:batch}. The dimension of the matrix is $200\times 200$, rank $r=5$, the probability to observe each entry is $0.1$. The sparse noise is sampled uniformly from $[-3,3]$. Each entry is corrupted by noise with probability $0.05$. We set $\sigma=1$, $R=5$ in problem \eqref{equ.matrix_completion}. We observe that our algorithm \ref{alg:non-convex-CGS} (NCGS) is much better than the non-convex Frank-Wolfe method (NFW). In Figure \ref{fig:finite-sum}, we treat problem \eqref{equ.matrix_completion} as a finite-sum problem, thus solve it using algorithm \ref{alg:non-convex-finite-sum-SCGS} (NCGS-VR) and  compare it with the result  of SVFW \citep{reddi2016stochasticfw}. We set the dimension of the matrix as $400\times 400$, $rank~r=8$, $\sigma=1$, $R=8$. The way to generate sparse noise and the probability to observe the entry are same with the setting of Figure \ref{fig:batch}. We observe that our NCGS-VR uses around 50 cpu-time to achieve $10^{-3}$ accuracy of squared gradient mapping, while SVFW needs more than 300 cpu-time.

\begin{figure}[h]
	\begin{subfigure}{0.45\textwidth}
		\includegraphics[width=\textwidth]{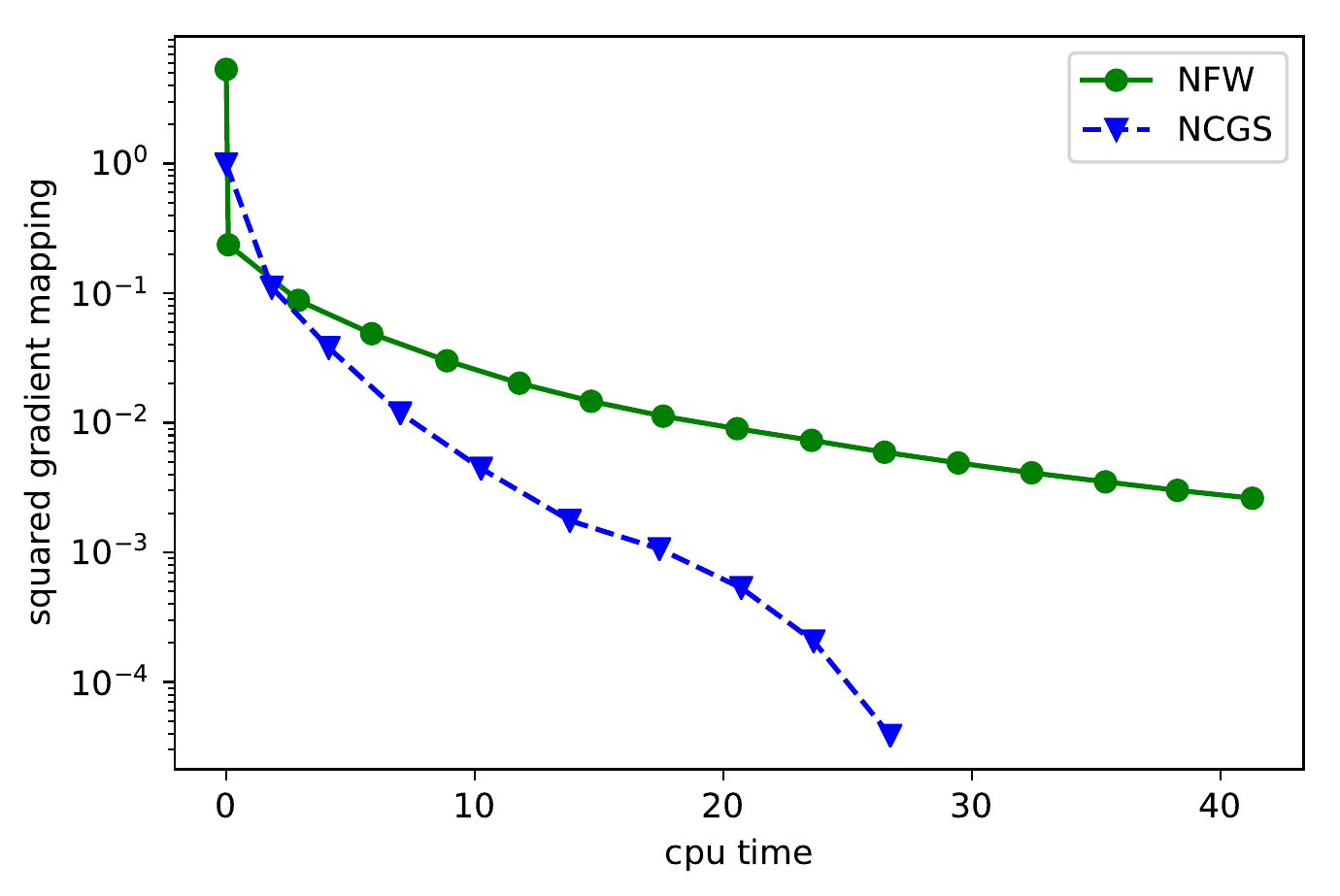}
		\caption{non-convex Frank-Wolfe and NCGS} \label{fig:batch}
	\end{subfigure}
	\begin{subfigure}{0.45\textwidth}
		\includegraphics[width=\textwidth]{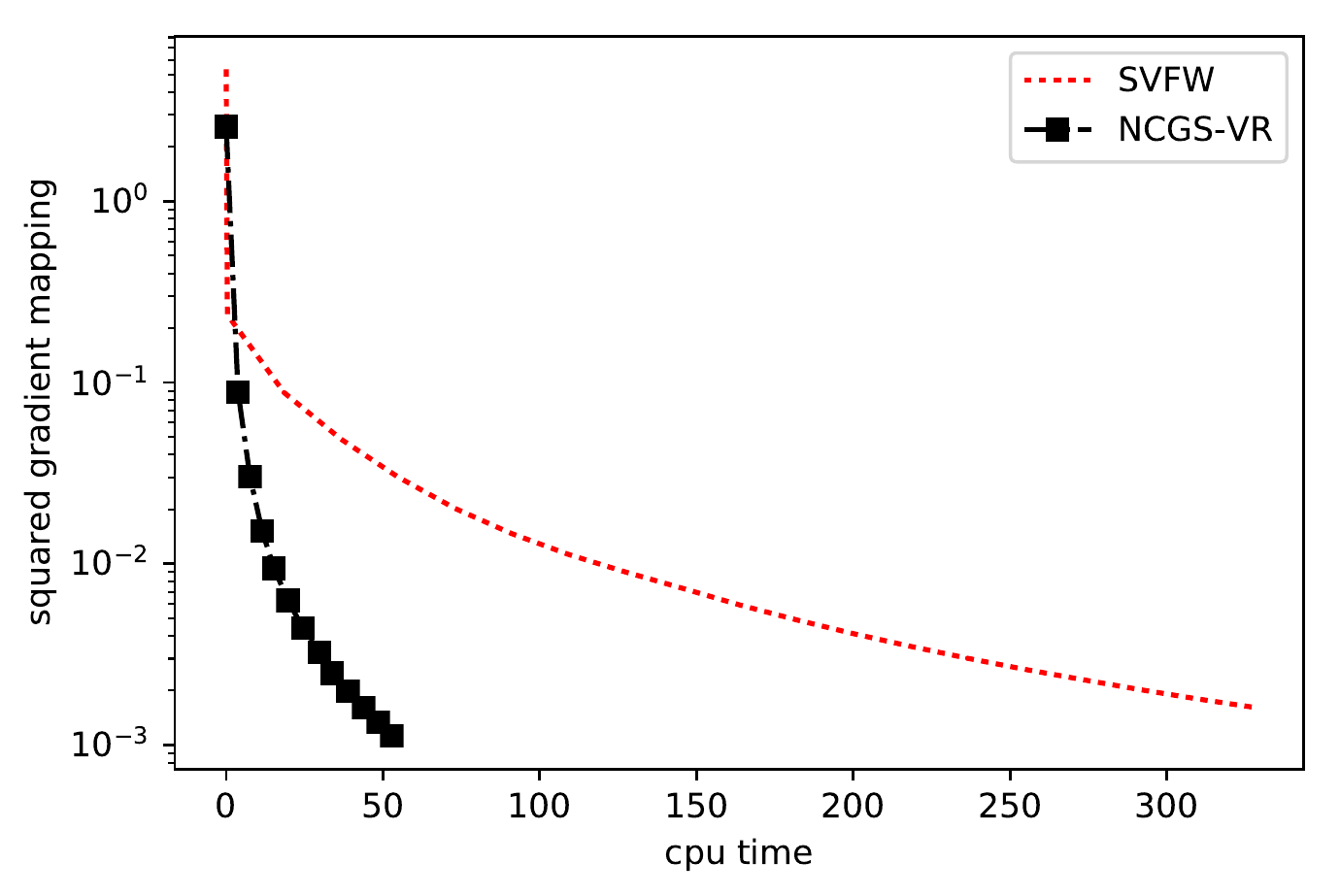}
		\caption{SVRFW and NCGS-VR}\label{fig:finite-sum}
	\end{subfigure}
	\caption{Frank-Wolfe method and non-convex conditional gradient sliding method in batched and finite-sum setting. The x-axis is the cpu-time, y-axis is the squared gradient mapping.}\label{fig:matrix_completion}
\end{figure}

\section{Conclusion}

In this paper, we propose the non-convex conditional gradient sliding method to solve the batch, stochastic and finite-sum non-convex problem with complex constraint. Our algorithms surpass  state of the art Frank-Wolfe type method  both theoretically and empirically.

\bibliography{NCGS}	
	\bibliographystyle{plainnat}
	
\newpage
\appendix
\section{Proof of Theorems and Corollaries }
In this section, we present all proofs of theorems and corollaries.

\subsection{Proof of Batched Setting}
We start with the proof of Theorem \ref{Theorem.batch}. 
\begin{proof}[proof of option I in Theorem \ref{Theorem.batch}]
	Define $\Delta_k=\nabla F(\theta_{k-1})-\nabla F(\theta_k^{md})$
	\begin{equation}\label{equ.using_smooth}
	\begin{split}
	F(\theta_k) &\leq F(\theta_{k-1})+\langle \nabla F(\theta_{k-1}),\theta_k-\theta_{k-1} \rangle+\frac{L}{2} \|\theta_k-\theta_{k-1}\|^2\\
	&=F(\theta_{k-1})+\langle \nabla F(\theta_k^{md}), \theta_k-\theta_{k-1}  \rangle+\frac{L}{2} \|\theta_k-\theta_{k-1}\|^2+\langle \Delta_k,\theta_k-\theta_{k-1} \rangle 
	\end{split}
	\end{equation}
	
	Now we use the termination condition of procedure  $condg$
	
	Recall we have  $\theta_{k}= condg(\nabla F(\theta_{k-1}^{md}),\theta_{k-1},\lambda_k,\eta_k). $
	
	The termination condition is 
	
	$$ \langle \nabla F(\theta_{k-1}^{md})+\frac{1}{\lambda_k} (\theta_{k}-\theta_{k-1}),\theta_k-u \rangle\leq \eta_k, \forall u\in \Omega.$$
	We choose $u=\theta_{k-1}$ and have
	
	\begin{equation}\label{equ.FW-gap}
	\langle \nabla F(\theta_{k-1}^{md
	}),\theta_{k}-\theta_{k-1} \rangle\leq -\frac{1}{\lambda_k} \|\theta_k-\theta_{k-1}\|^2+\eta_k.
	\end{equation} 
	
	Now we substitute the upper bound of $ \langle \nabla F(\theta_{k-1}^{md
	}),\theta_{k}-\theta_{k-1} \rangle $ in \eqref{equ.FW-gap} for the terms in \eqref{equ.using_smooth} and get
	
	\begin{equation}\label{equ.mid1}
	\begin{split}
	F(\theta_k)&\leq F(\theta_{k-1})-\frac{1}{\lambda_k}\|\theta_k-\theta_{k-1}\|^2+\frac{L}{2}\|\theta_{k}-\theta_{k-1}\|^2+\langle \Delta_k,\theta_k-\theta_{k-1} \rangle +\eta_k\\
	&\leq F(\theta_{k-1})-\frac{1}{\lambda_k}\|\theta_k-\theta_{k-1}\|^2+\frac{L}{2}\|\theta_{k}-\theta_{k-1}\|^2+\|\Delta_k\|\|\theta_k-\theta_{k-1}\|+\eta_k.
	\end{split}
	\end{equation}
	where the second inequality holds from the Cauchy-Schwarz inequality.
	
	Now we prepare to bound term $\|\Delta_k\|$, recall that $\Delta_k=\nabla F(\theta_{k-1})-\nabla F(\theta_{k}^{md})$.
	
	$$\|\Delta_k\|=\|\nabla F(\theta_{k-1})-\nabla F(\theta_{k}^{md})\|\leq L\|\theta_{k-1}-\theta_{k}^{md} \|=L(1-\alpha_k) \|\theta_{k-1}^{ag}-\theta_{k-1}\|.$$
	Replace $\|\Delta_k\|$ by this upper bound in \eqref{equ.mid1}, we get
	
	\begin{equation}\label{equ.mid2}
	\begin{split}
	F(\theta_k)&\leq F(\theta_{k-1})-\frac{1}{\lambda_k}\|\theta_k-\theta_{k-1}\|^2+\frac{L}{2}\|\theta_{k}-\theta_{k-1}\|^2+L(1-\alpha_k) \|\theta_{k-1}^{ag}-\theta_{k-1}\|\|\theta_k-\theta_{k-1}\|+\eta_k\\
	&\leq F(\theta_{k-1})+(\frac{L}{2}-\frac{1}{\lambda_k})\|\theta_k-\theta_{k-1}\|^2+\frac{L}{2}\|\theta_{k}-\theta_{k-1}\|^2+\frac{L(1-\alpha_{k})^2}{2}\|\theta_{k-1}^{ag}-\theta_{k-1}\|^2+\eta_k\\
	&\leq F(\theta_{k-1}) +(L-\frac{1}{\lambda_k})\|\theta_k-\theta_{k-1}\|^2+\frac{L(1-\alpha_k)^2}{2} \|\theta_{k-1}^{ag}-\theta_{k-1}\|^2+\eta_k,
	\end{split}
	\end{equation}
	where the second inequity holds from the fact $a^2+b^2\geq 2ab$.
	
	Recall the definition of approximated gradient mapping, so we have
	\begin{equation}
	\begin{split}
	&\theta_k^{ag}-\theta_k\\
	= &\theta_{k}^{md}-\beta_k \tilde{g}(\theta_{k-1},\nabla F(\theta_{k}^{md}),\lambda_k,\eta_k)-\theta_k\\
	=&(1-\alpha_k)\theta_{k-1}^{ag} +\alpha_k \theta_{k-1}-\beta_k \tilde{g}(\theta_{k-1},\nabla F(\theta_{k}^{md}),\lambda_k,\eta_k)- \big( \theta_{k-1}-\lambda_k \tilde{g}(\theta_{k-1},\nabla F(\theta_{k}^{md}),\lambda_k,\eta_k)  \big)\\
	=&(1-\alpha) (\theta_{k-1}^{ag}-\theta_{k-1})+(\lambda_k-\beta_k) \tilde{g}(\theta_{k-1},\nabla F(\theta_{k}^{md}),\lambda_k,\eta_k).
	\end{split}
	\end{equation}
	
	Now we apply Lemma \ref{lemma.Gamma} on $\theta_{k}^{ag}-\theta_k$ and have 
	
	$$ \theta_{k}^{ag}-\theta_k=\Gamma_{k} \sum_{\tau=1}^{k} (\frac{\lambda_\tau-\beta_\tau}{\Gamma_{\tau}}) \tilde{g}(\theta_{\tau-1},\nabla F(\theta_{\tau}^{md}),\lambda_{\tau},\eta_{\tau}). $$

	Now using Jensens's inequality and the fact that 
	
	$$ \sum_{\tau=1}^{k} \frac{\alpha_{\tau}}{\Gamma_{\tau}} =\frac{\alpha_1}{\Gamma_1}+\sum_{\tau=2}^{k} \frac{1}{\Gamma_{\tau}} (1-\frac{\Gamma_{\tau}}{\Gamma_{\tau-1}})=\frac{1}{\Gamma_k},$$ we have
	
	\begin{equation}
	\begin{split}
	\|\theta_{k}^{ag}-\theta_{k}\|^2&=\| \Gamma_{k} \sum_{\tau=1}^{k} (\frac{\lambda_\tau-\beta_\tau}{\Gamma_{\tau}}) \tilde{g}(\theta_{\tau-1},\nabla F(\theta_{k}^{md}),\lambda_{\tau},\eta_{\tau}) \|^2\\
	&=\| \Gamma_{k} \sum_{\tau=1}^{k} \frac{\alpha_{\tau}}{\Gamma_\tau} (\frac{\lambda_\tau-\beta_\tau}{\alpha_{\tau}}) \tilde{g}(\theta_{\tau-1},\nabla F(\theta_{k}^{md}),\lambda_{\tau},\eta_{\tau}) \|^2\\
	&\leq \Gamma_k \sum_{\tau=1}^{k} \frac{\alpha_{\tau}}{\Gamma_{\tau}} \|(\frac{\lambda_\tau-\beta_\tau}{\alpha_{\tau}}) \tilde{g}(\theta_{\tau-1},\nabla F(\theta_{k}^{md}),\lambda_{\tau},\eta_{\tau})\|^2\\
	&=\Gamma_k \sum_{\tau=1}^{k} \frac{(\lambda_{\tau}-\beta_\tau)^2}{\Gamma_{\tau}\alpha_{\tau}} \| \tilde{g}(\theta_{\tau-1},\nabla F(\theta_{k}^{md}),\lambda_{\tau},\eta_{\tau}) \|^2.
	\end{split}
	\end{equation}
	
	Now replace above upper bound in \eqref{equ.mid2} we have 
	
	\begin{equation}
	\begin{split}
	F(\theta_{k})&\leq F(\theta_{k-1})-\lambda_k (1-L\lambda_k)\| \tilde{g}(\theta_{k-1},\nabla F(\theta_{k}^{md}),\lambda_k,\eta_k)\|^2\\
	&+\frac{L(1-\alpha_{k})^2}{2} \Gamma_k \sum_{\tau=1}^{k-1} \frac{(\lambda_{\tau}-\beta_\tau)^2}{\Gamma_{\tau}\alpha_{\tau}} \| \tilde{g}(\theta_{\tau-1},\nabla F(\theta_{k}^{md}),\lambda_{\tau},\eta_{\tau}) \|^2+\eta_k\\
	&\leq F(\theta_{k-1})-\lambda_k (1-L\lambda_k)\| \tilde{g}(\theta_{k-1},\nabla F(\theta_{k}^{md}),\lambda_k,\eta_k)\|^2+\\
	&\frac{L\Gamma_{k}}{2} \sum_{\tau=1}^{k} \frac{(\lambda_{\tau}-\beta_\tau)^2}{\Gamma_{\tau}\alpha_{\tau}} \| \tilde{g}(\theta_{\tau-1},\nabla F(\theta_{k}^{md}),\lambda_{\tau},\eta_{\tau}) \|^2+\eta_k.
	\end{split}
	\end{equation}

	Now sum over both side, we obtain
	
	\begin{equation}
	\begin{split}
	F(\theta_{N})&\leq F(\theta_0)-\sum_{k=1}^{N} \lambda_k (1-L\lambda_k)\| \tilde{g}(\theta_{k-1},\nabla F(\theta_{k}^{md}),\lambda_k,\eta_k)\|^2\\
	&+\sum_{k=1}^{N}\frac{L\Gamma_{k}}{2} \sum_{\tau=1}^{k} \frac{(\lambda_{\tau}-\beta_\tau)^2}{\Gamma_{\tau}\alpha_{\tau}} \| \tilde{g}(\theta_{\tau-1},\nabla F(\theta_{k}^{md}),\lambda_{\tau},\eta_{\tau}) \|^2
	+\sum_{k=1}^{N}\eta_k\\
	&=F(\theta_0)-\sum_{k=1}^{N} \lambda_k (1-L\lambda_k)\| \tilde{g}(\theta_{k-1},\nabla F(\theta_{k}^{md}),\lambda_k,\eta_k)\|^2\\
	& +\frac{L}{2}\sum_{k=1}^{N} \frac{(\lambda_k-\beta_k)^2}{\Gamma_{k}\alpha_{k}} \big( \sum_{\tau=k}^{N} \Gamma_\tau\big)\|\tilde{g}(\theta_{k-1},\nabla F(\theta_{k}^{md}),\lambda_k,\eta_k)\|^2
	+\sum_{k=1}^{N}\eta_k\\
	&=F(\theta_0)-\sum_{k=1}^{N}\lambda_kC_k \|\tilde{g}(\theta_{k-1},\lambda_k,\eta_k)\|^2+\sum_{k=1}^{N}\eta_k,
	\end{split}
	\end{equation}
	
	where $C_k=1-L\lambda_k-\frac{L(\lambda_k-\beta_k)^2}{2\alpha_k\Gamma_k\lambda_k} (\sum_{\tau=k}^{N}\Gamma_\tau).$

	Now rearrange terms and using the fact that $F(\theta^*)\leq F(\theta_N)$, we obtain
	
	$$ \min_{k=1,...,N}  \|\tilde{g}(\theta_{k-1},\nabla F(\theta_{k}^{md}),\lambda_k,\eta_k)\|^2 \big( \sum_{k=1}^{N} \lambda_kC_k  \big)\leq F(\theta_0)-F(\theta^*)+\sum_{k=1}^{N}\eta_k.$$

	Recall $\alpha_k=\frac{2}{k+1}$, $\beta_k=1/(2L),\lambda_k\in [\beta_k, (1+\frac{\alpha_k}{4})\beta_k],\eta_k=\frac{1}{N}$. In this setting, we have

	\begin{equation}\label{equ.gamma_relation}
	\Gamma_k=\frac{2}{k(k+1)}, \sum_{\tau=k}^{N}\Gamma_\tau=\sum_{\tau=k}^{N}\frac{2}{\tau(\tau+1)}\leq \frac{2}{k}.
	\end{equation}

	Since $0\leq \lambda_k-\beta_k\leq \alpha_k\beta_k/4$.
	
	We got $C_k\geq 1-L\big[(1+\frac{\alpha_k}{4})\beta_k+\frac{\alpha_k^2\beta^2_k}{16k\alpha_k\Gamma_k\beta_k}  \big]=1-\beta_kL (1+\frac{\alpha_k}{4}+\frac{1}{16})\geq 1-\beta_kL\frac{21}{16}=11/32$.
	
	So we have $\lambda_kC_k\geq \frac{11\beta_k}{32}\geq \frac{1}{6L}$.
	
	$$ \min_{k=1,...,N}  \|\tilde{g}(\theta_{k-1},\nabla F(\theta_k^{md}),\lambda_k,\eta_k)\|^2\leq \frac{6L (F(\theta_0))-F(\theta^*) +1)}{N}$$

	The next step is to bound the distance of approximated gradient mapping and true gradient mapping, i.e., $$\|\tilde{g}(\theta_{k-1},\nabla F(\theta_{k}^{md}),\lambda_k,\eta_k)-g(\theta_{k-1},\nabla F(\theta_{k}^{md}),\lambda_k)\|^2.$$
	
	Using Lemma \ref{lemma.bound_gradient_mapping}, we obtain 
	
	$\|\tilde{g}(\theta_{k-1},\nabla F(\theta_{k}^{md}),\lambda_k,\eta_k)-g(\theta_{k-1},\nabla F(\theta_{k}^{md}),\lambda_k)\|^2\leq \eta_k/\lambda_k\leq \frac{2L}{N}$.
	
	Thus we have 
	
	$$\min_{k=1,...,N} \|g(\theta_{k-1},\nabla F(\theta_{k}^{md}),\lambda_k)\|^2\leq \frac{12L(F(\theta_0)-F(\theta^*))+16L}{N}. $$
	
\end{proof}

\begin{proof}[Proof of option II in Theorem \ref{Theorem.batch}]

	Note that the procedure  $condg (\nabla F(\theta_{k}^{md}),\theta_{k-1}, \lambda_k,\eta_k) $actually solve the following problem with tolerance $\eta_k$.
	
	$$ \min_{\theta\in \Omega} \langle \nabla F (\theta_k^{md}),\theta \rangle+\frac{1}{2\lambda_k}\|\theta_{k-1}-\theta\|^2. $$ 
	
	Using strong convexity of above objective function ($w.r.t.~\theta$), we have $\forall \theta\in \Omega $

	\begin{equation}
	\begin{split}
	&\langle \nabla F (\theta_k^{md}),\theta \rangle+\frac{1}{2\lambda_k}\|\theta_{k-1}-\theta\|^2- \langle \nabla F (\theta_k^{md}), \theta_k \rangle-\frac{1}{2\lambda_k}\|\theta_{k-1}-\theta_k\|^2\\
	\geq & \langle  \nabla F (\theta_k^{md})+\frac{1}{\lambda_k} (\theta_k-\theta_{k-1}),\theta-\theta_k \rangle+\frac{1}{2\lambda_k}\|\theta-\theta_k\|^2.
	\end{split}
	\end{equation}
	
	Recall the termination condition
	
	$$ \langle \nabla F(\theta_k^{md})+ \frac{1}{\lambda_k} (\theta_k-\theta_{k-1}) , \theta_k-\theta \rangle\leq \eta_k, \forall \theta\in \Omega, $$
	
	and rearrange terms, we have
	
	\begin{equation}\label{equ.optionII.mid}
	\langle \nabla F(\theta_k^{md}),\theta_k-\theta \rangle\leq \frac{1}{2\lambda_k} \big( \|\theta_{k-1}-\theta\|^2-\|\theta_{k}-\theta\|^2- \|\theta_{k-1}-\theta_k\|^2       \big)+\eta_k. 
	\end{equation}

	Apply same argument on $\theta_{k}^{ag}$, we have

	\begin{equation}\label{equ.optionII.mid2}
	\langle \nabla F(\theta_k^{md}),\theta_k^{ag}-\theta \rangle\leq \frac{1}{2\beta_k} \big( \|\theta_k^{md}-\theta\|^2-\|\theta_{k}^{ag}-\theta\|^2- \|\theta_{k}^{ag}-\theta_k^{md}\|^2       \big)+\chi_k. 
	\end{equation}

	Now choose $\theta=\alpha_k\theta_k+(1-\alpha_{k})\theta_{k-1}^{ag}$ in \eqref{equ.optionII.mid2} and we have

	\begin{equation}\label{equ.optionII.mid3}
	\begin{split}
	\langle F(\theta_k^{md}),\theta_k^{ag}-\alpha_k\theta_k-(1-\alpha_k)\theta_{k-1}^{ag} \rangle\leq &\frac{1}{2\beta_k} \big(  \|\theta_k^{md}-\alpha_{k} \theta_k-(1-\alpha_k) \theta_{k-1}^{ag}\|^2-\|\theta_k^{ag}-\theta_{k}^{md}\|^2 \big)\\
	=& \frac{1}{2\beta_k}\big( \alpha_k^2\| \theta_k-\theta_{k-1} \|^2-\|\theta_k^{ag}-\theta_k^{md}\|^2      \big)+\chi_k.
	\end{split}
	\end{equation}
	
	Add \eqref{equ.optionII.mid3} and $\alpha_{k}\times$\eqref{equ.optionII.mid}, we have
	
	\begin{equation}\label{equ.optionII.mid4}
	\begin{split}
	&\langle \nabla F(\theta_k^{md}), \theta_{k}^{ag}-\alpha_{k}\theta-(1-\alpha_k)\theta_{k-1}^{ag} \rangle\\
	\leq & \frac{\alpha_k}{2\lambda_k} (\|\theta_{k-1}-\theta\|^2-\|\theta_{k}-\theta\|^2)+\frac{\alpha_{k} (\lambda_k\alpha_k-\beta_k)}{2\beta_k\lambda_k}\|\theta_k-\theta_{k-1}\|^2-\frac{1}{2\beta_k}\|\theta_k^{ag}-\theta_k^{md}\|^2\\
	\leq & \frac{\alpha_k}{2\lambda_k} (\|\theta_{k-1}-\theta\|^2-\|\theta_{k}-\theta\|^2)-\frac{1}{2\beta_k}\|\theta_k^{ag}-\theta_k^{md}\|^2+\alpha_k\eta_k+\chi_k,
	\end{split}
	\end{equation}
	where the last inequality uses the assumption $\lambda_k\alpha_k\leq \beta_k$.
	
	Note that	
	\begin{equation}\label{equ.optionII.mid5}
	\begin{split}
	&\alpha_k (F(\theta_k^{md})-F(\theta))+(1-\alpha_k) (F(\theta_k^{md})-F(\theta_{k-1}^{ag}))\\
	\leq & \alpha_k (\langle \nabla F(\theta_{k}^{md}), \theta_k^{md}-\theta     \rangle+\frac{l}{2}\|\theta-\theta_k^{md
	}\|^2 )+ (1-\alpha_{k})\big(  \langle \nabla F(\theta_{k}^{md}), \theta_k^{md}-\theta_{k-1}^{ag} \rangle+\frac{l}{2}\|\theta_{k}^{md}-\theta_{k-1}^{ag}\|^2   \big)\\
	=& \langle \nabla F(\theta_k^{md}),\theta_k^{md}-\alpha_{k}\theta-(1-\alpha_{k})\theta_{k-1}^{ag}   \rangle+ \frac{\alpha_{k}l}{2} \|\theta-\theta_k^{md}\|^2 +\frac{(1-\alpha_k)l}{2}\|\theta_k^{md}-\theta_{k-1}^{ag}\|^2\\
	=&  \langle \nabla F(\theta_k^{md}),\theta_k^{md}-\alpha_{k}\theta-(1-\alpha_{k})\theta_{k-1}^{ag}   \rangle+ \frac{\alpha_{k}l}{2} \|\theta-\theta_k^{md}\|^2 + \frac{\alpha_k^2(1-\alpha_k)l}{2}\|\theta_{k-1}^{ag}-\theta_{k-1}\|^2,
	\end{split}
	\end{equation}

	and
	\begin{equation}
	F(\theta_k^{ag})\leq F(\theta_k^{md})+\langle \nabla F(\theta_k^{md}), \theta_{k}^{ag}-\theta_k^{md} \rangle+\frac{L}{2}\|\theta_k^{ag}-\theta_k^{md}\|^2.
	\end{equation}
	
	Combine above equation with \eqref{equ.optionII.mid4} and \eqref{equ.optionII.mid5}, we have

	\begin{equation}
	\begin{split}
	&F(\theta_k^{ag})-F(\theta)\leq (1-\alpha_k) F(\theta_{k-1}^{ag})-(1-\alpha_{k})F(\theta) -\frac{1}{2}(\frac{1}{\beta_k}-L) \|\theta_{k}^{ag}-\theta_{k}^{md}\|^2\\
	&+\frac{\alpha_k}{2\lambda_k} ( \|\theta_{k-1}-\theta\|^2-\|\theta_{k}-\theta\|^2 )+\frac{l\alpha_{k}}{2} \|\theta_{k}^{md}-\theta\|^2+\frac{l\alpha_{k}^2 (1-\alpha)_k}{2} \|\theta_{k-1}^{ag}-\theta_{k-1}\|^2+\alpha_k\eta_k+\chi_k.
	\end{split}
	\end{equation}
	
	Now we apply Lemma \ref{lemma.Gamma} and have
	
	\begin{equation}
	\begin{split}
	&\frac{F(\theta_N^{ag})-F(\theta)}{\Gamma_N}+\sum_{k=1}^{N} \frac{1-L\beta_k}{2\beta_k\Gamma_k} \|\theta_k^{ag}-\theta_k^{md}\|^2\\
	\leq &\frac{\|\theta_0-\theta\|^2}{2\lambda_1}+\frac{l}{2} \sum_{k=1}^{N} \frac{\alpha_k}{\Gamma_k} \big( \|\theta_k^{md}-\theta\|^2+\alpha_k(1-\alpha_k)\|\theta_{k-1}^{ag}-\theta_{k-1}\|^2 \big)+\sum_{k=1}^{N}\frac{\alpha_k}{\Gamma_k} \eta_k+\sum_{k=1}^{N}\frac{\chi_k}{\Gamma_k}.   
	\end{split}
	\end{equation}
	
	Now set $\theta=\theta^*$ in above equation, and notice 
	
	\begin{equation}
	\begin{split}
	&\|\theta_k^{md}-\theta^*\|^2+\alpha_k(1-\alpha_k)\|\theta_{k-1}^{ag}-\theta_{k-1}\|^2\\
	\leq &2\big( \|\theta^*\|+\|\theta^{md}_k\|^2+ \alpha_k(1-\alpha_k) ( \|\theta^{ag}_{k-1}\|^2 +\|\theta_{k-1}\|^2)    \big)\\
	\leq &2 \big( \|\theta^*\|^2+(1-\alpha_k) \|\theta_{k-1}^{ag}\|^2+\alpha_k \|\theta_{k-1}\|^2+\alpha_k(1-\alpha_k) ( \|\theta^{ag}_{k-1}\|^2 +\|\theta_{k-1}\|^2)     \big)\\
	\leq & 2 (\|\theta^*\|^2+\|\theta_{k-1}^{ag}\|^2+\|\theta_{k-1}\|^2)\\
	\leq &2(\|\theta^*\|^2+2M^2),
	\end{split}
	\end{equation}
	
	and recall the definition of $\Gamma_k$ and \eqref{equ.gamma_relation} we obtain

	\begin{equation}
	\begin{split}
	&\frac{F(\theta_N^{ag})-F(\theta)}{\Gamma_N}+\sum_{k=1}^{N} \frac{1-L\beta_k}{2\beta_k\Gamma_k} \|\theta_k^{ag}-\theta_k^{md}\|^2\\\leq &\frac{\|\theta_0-\theta^*\|^2}{2\lambda_1} +l\sum_{k=1}^{N} \frac{\alpha_k}{\Gamma_k} (\|\theta^*\|^2+2M^2)+\sum_{k=1}^{N}\frac{\chi_k}{\Gamma_k}+\sum_{k=1}^{N}\frac{\alpha_k}{\Gamma_k} \eta_k\\
	\leq &  \frac{\|\theta_0-\theta^*\|^2}{2\lambda_1}+\frac{l}{\Gamma_N} (\|\theta^*\|^2+2M^2)+\sum_{k=1}^{N}\frac{\chi_k}{\Gamma_k}+\sum_{k=1}^{N}\frac{\alpha_k}{\Gamma_k} \eta_k.
	\end{split}
	\end{equation}
	
	Now using the setup of $ \chi_k,\eta_k, \gamma_k$ in the theorem \ref{Theorem.stochastic} we obtain
	
	$$ \min_{k=1,...,N} \|\theta_k^{ag}-\theta_k^{md}\|^2 \leq \frac{6}{L} \big(  \frac{4L \|\theta_0-\theta^*\|^2}{N^2(N+1)}+\frac{l}{N} (\|\theta^*\|^2+2M^2) +\frac{1}{N} \big). $$
	
	Recall the definition of approximated gradient mapping $ \tilde{g}(\theta_{k-1},\nabla F(\theta_{k}^{md}),\beta_k,\chi_k):= \frac{\theta^{md}_k-\theta_k^{ag}}{\beta_k}.$

	Thus  $\min_{k=1,...,N}  \|\tilde{g}(\theta_{k-1},\nabla F(\theta_{k}^{md}),\beta_k,\chi_k)\|^2\leq 24L \big(  \frac{4L \|\theta_0-\theta^*\|^2}{N^2(N+1)}+\frac{l}{N} (\|\theta^*\|^2+2M^2)+\frac{1}{N}  \big).$

	Using Lemma \ref{lemma.bound_gradient_mapping}, we have
	
	$$ \|\tilde{g}(\theta_{k-1},\nabla F(\theta_{k}^{md}),\beta_k,\chi_k)-g(\theta_{k-1},\nabla F(\theta_{k}^{md}),\beta_k)\|^2 \leq \chi_k/\beta_k ,$$
	thus we obtain
	
	$$ \min_{k=1,...,N} \|g(\theta_{k-1},\nabla F(\theta_{k}^{md}),\beta_k)\|^2\leq  48L \big(  \frac{4L \|\theta_0-\theta^*\|^2}{N^2(N+1)}+\frac{l}{N} (\|\theta^*\|^2+2M^2)+\frac{2}{N}  \big). $$

\end{proof}

\begin{proof}[Proof of corollary \ref{Corollary.batch}]
	
	Note that the procedure  $condg (\nabla F(\theta_{k}^{md}),\theta_{k-1}, \lambda_k,\eta_k) $	actually solves the following problem using frank-wolfe method with tolerance $\eta_k$. In option I, In each call of condg, we need $\frac{1}{2\lambda_k}/\eta_k=NL$ steps to converge with tolerance $\eta_k$ according to the standard proof of Frank-Wolfe method. Thus the total number of LO is $\mathcal{O} (N^2).$  Similarly, in option II, we have two calls of condg, where they need $\frac{1}{2\lambda_k}/\eta_k= \lceil \frac{2NL}{k} \rceil$ and  $\frac{1}{2\beta_k}/\chi_k=NL$ steps to converges with tolerance $\eta_k$ and $\chi_k$. Thus the total number of LO is $\mathcal{O} (N^2).$
	
\end{proof}

\begin{lemma}\label{lemma.bound_gradient_mapping}
	$\|\tilde{g}(\theta_{k-1},\nabla F(\theta_{k}^{md}),\lambda_k,\eta_k)-g(\theta_{k-1},\nabla F(\theta_{k}^{md}),\lambda_k)\|^2\leq \eta_k/\lambda_k$, where $\lambda_k$ is the stepsize in the algorithm, $\eta_k$ the tolerance in the procedure $condg$.
\end{lemma}

\begin{proof}
	Define $ \tilde{\theta}=condg(l,u,\lambda,\eta)$, $\hat{\theta}=\arg\min_{x\in \Omega} \langle l, x  \rangle+\frac{1}{2\lambda}\|x-\mu\|^2 $.
	
	Using the termination condition of the procedure, we have
	$$ \langle l+\frac{1}{\lambda}(\tilde{\theta}-u),\tilde{\theta}-x \rangle\leq \eta, \forall x\in \Omega. $$
	
	Now we choose $x=\hat{\theta}$ and rearrange the therm then we have
	
	$$\langle l+\frac{1}{\lambda}(\hat{\theta}-u),\tilde{\theta}-\hat{\theta} \rangle+\frac{1}{\lambda}\|\tilde{\theta}-\hat{\theta}\|^2\leq \eta.$$
	
	Notice $\langle l+\frac{1}{\lambda}(\hat{\theta}-u),\tilde{\theta}-\hat{\theta} \rangle\geq 0 $ by the optimal condition of $\hat{\theta}$. Thus we have $\|\tilde{\theta}-\hat{\theta}\|^2\leq  \eta\lambda $, i.e.,  $\| (\tilde{\theta}-\hat{\theta})/\lambda\|^2\leq  \eta/\lambda.$
	
\end{proof}

\begin{lemma}\label{lemma.Gamma}
	Let $\alpha_{k}$ be the stepsize in the algorithm \ref{Theorem.stochastic} option II, and the sequence $\{ h_k\}$ satisfies 
	\begin{equation}\label{equ.recursive}
	h_k \leq (1-\alpha_k) h_{k-1} +\psi_k,~ k=1,2,...,
	\end{equation}

	then we have $h_k\leq \Gamma_k\sum_{i=1}^{k}(\psi_i/\Gamma_i)$ for any $k\geq 1$, where 
	
	$$\Gamma_k=\begin{cases}
	1, & k=1\\
	(1-\alpha_k) \Gamma_{k-1} & k\geq 2\\
	\end{cases} 
	$$
\end{lemma}
\begin{proof}
	Notice $\alpha_1=1$ and $\alpha \in (0,1)$ for $k\geq 2$ and then divide both side of \eqref{equ.recursive} by $\Gamma_k$, we have
	
	$$ \frac{h_1}{\Gamma_1}\leq \frac{(1-\alpha_1)h_1}{\Gamma_1}+\frac{\psi_1}{\Gamma_1}= \frac{\psi_1}{\Gamma_1}$$
	and
	
	$$ \frac{h_k}{\Gamma_k} \leq \frac{(1-\alpha_k)h_{k-1}}{\Gamma_k}+\frac{\psi_k}{\Gamma_k}=\frac{h_{k-1}}{\Gamma_{k-1}}+\frac{\psi_k}{\Gamma_k}, ~~for~~ k\geq 2.$$
	
	Sum over both side, we have the result. 
\end{proof}

\subsection{Proof of Stochastic Setting}

\begin{proof}[Proof of Theorem \ref{Theorem.stochastic}]
We denote $\bar{\delta}_k:=\bar{G}_k-\nabla \Psi(\theta_k^{md})$ and $\bar{\delta}_{[k]}:={\bar{\delta}_1,...,\bar{\delta}_k} .$	Similar to the batched case, the procedure $condg$ solve the following problem with tolerance $\eta_k$.

	$$ \min_{x\in \Omega} \langle  \bar{G}_k,x \rangle+\frac{1}{2\lambda_k}\|\theta_{k-1}-x\|^2  $$

	Again, use the strong convexity of objective function (w.r.t.  x) we have

	\begin{equation}
	\begin{split}
	&\langle  \bar{G}_k,\theta \rangle+\frac{1}{2\lambda_k}\|\theta_{k-1}-\theta\|^2- \langle \bar{G}_k, \theta_k \rangle-\frac{1}{2\lambda_k}\|\theta_{k-1}-\theta_k\|^2\\
	\geq & \langle  \bar{G}_k+\frac{1}{\lambda_k} (\theta_k-\theta_{k-1}),\theta-\theta_k \rangle+\frac{1}{2\lambda_k}\|\theta-\theta_k\|^2.
	\end{split}
	\end{equation}
	
	Recall the termination condition
	
	$$ \langle \bar{G}_k+ \frac{1}{\lambda_k} (\theta_k-\theta_{k-1}) , \theta_k-\theta \rangle\leq \eta_k, \forall \theta\in \Omega $$
	
	and rearrange terms, we have
	
	\begin{equation}\label{equ.SCGS.mid}
	\langle \bar{G}_k,\theta_k-\theta \rangle\leq \frac{1}{2\lambda_k} \big( \|\theta_{k-1}-\theta\|^2-\|\theta_{k}-\theta\|^2- \|\theta_{k-1}-\theta_k\|^2       \big)+\eta_k. 
	\end{equation}
	
	We have similar result on $\theta_{k}^{ag}$, i.e.,

	\begin{equation}\label{equ.SCGS.mid2}
	\langle \bar{G}_k,\theta_k^{ag}-\theta \rangle\leq \frac{1}{2\beta_k} \big( \|\theta_k^{md}-\theta\|^2-\|\theta_{k}^{ag}-\theta\|^2- \|\theta_{k}^{ag}-\theta_k^{md}\|^2       \big)+\chi_k. 
	\end{equation}

	Now choose $\theta=\alpha_k\theta_k+(1-\alpha_{k})\theta_{k-1}^{ag}$ in \eqref{equ.optionII.mid2} and we have

	\begin{equation}\label{equ.SCGS.mid3}
	\begin{split}
	\langle \bar{G}_k,\theta_k^{ag}-\alpha_k\theta_k-(1-\alpha_k)\theta_{k-1}^{ag} \rangle\leq &\frac{1}{2\beta_k} \big(  \|\theta_k^{md}-\alpha_{k} \theta_k-(1-\alpha_k) \theta_{k-1}^{ag}\|^2-\|\theta_k^{ag}-\theta_{k}^{md}\|^2 \big)\\
	=& \frac{1}{2\beta_k}\big( \alpha_k^2\| \theta_k-\theta_{k-1} \|^2-\|\theta_k^{ag}-\theta_k^{md}\|^2      \big)+\chi_k.
	\end{split}
	\end{equation}
	
	Add \eqref{equ.SCGS.mid3} and $\alpha_{k}\times$\eqref{equ.SCGS.mid} together and recall the definition of $\bar{G}_k=\nabla F(\theta_k^{md
	})+\bar{\delta}_k$ we have
	
	\begin{equation}\label{equ.SCGS.mid4}
	\begin{split}
	&\langle \nabla F(\theta_k^{md})+\bar{\delta}_k, \theta_{k}^{ag}-\alpha_{k}\theta-(1-\alpha_k)\theta_{k-1}^{ag} \rangle\\
	\leq & \frac{\alpha_k}{2\lambda_k} (\|\theta_{k-1}-\theta\|^2-\|\theta_{k}-\theta\|^2)+\frac{\alpha_{k} (\lambda_k\alpha_k-\beta_k)}{2\beta_k\lambda_k}\|\theta_k-\theta_{k-1}\|^2-\frac{1}{2\beta_k}\|\theta_k^{ag}-\theta_k^{md}\|^2\\
	\leq & \frac{\alpha_k}{2\lambda_k} (\|\theta_{k-1}-\theta\|^2-\|\theta_{k}-\theta\|^2)-\frac{1}{2\beta_k}\|\theta_k^{ag}-\theta_k^{md}\|^2+\alpha_k\eta_k+\chi_k,
	\end{split}
	\end{equation}
	where the last inequality uses the assumption $\lambda_k\alpha_k\leq \beta_k$.
	
	Again use the smoothness of objective function $F(\theta)$, 	\begin{equation}\label{equ.SCGS.smooth}
	F(\theta_k^{ag})\leq F(\theta_k^{md})+\langle \nabla F(\theta_k^{md}), \theta_{k}^{ag}-\theta_k^{md} \rangle+\frac{L}{2}\|\theta_k^{ag}-\theta_k^{md}\|^2.
	\end{equation}
	
	Combine \eqref{equ.SCGS.smooth}, \eqref{equ.optionII.mid5} and \eqref{equ.SCGS.mid4} together, we obtain

	\begin{equation}
	\begin{split}
	&F(\theta_k^{ag})-F(\theta)\leq (1-\alpha_k) F(\theta_{k-1}^{ag})-(1-\alpha_{k})F(\theta) -\frac{1}{2}(\frac{1}{\beta_k}-L) \|\theta_{k}^{ag}-\theta_{k}^{md}\|^2\\
	&+\frac{\alpha_k}{2\lambda_k} ( \|\theta_{k-1}-\theta\|^2-\|\theta_{k}-\theta\|^2 )+\frac{l\alpha_{k}}{2} \|\theta_{k}^{md}-\theta\|^2+\frac{l\alpha_{k}^2 (1-\alpha)_k}{2} \|\theta_{k-1}^{ag}-\theta_{k-1}\|^2+\\
	&+\langle \bar{\delta}_k, \alpha_k (\theta-\theta_{k-1})+\theta_k^{md}-\theta_k^{ag} \rangle
	+\alpha_k\eta_k+\chi_k\\
	&\leq (1-\alpha_k) F(\theta_{k-1}^{ag})-(1-\alpha_{k})F(\theta) -\frac{1}{4}(\frac{1}{\beta_k}-L) \|\theta_{k}^{ag}-\theta_{k}^{md}\|^2+\frac{\beta_k \|\bar{\delta}_k\|^2 }{1-L\beta_k}\\
	&+\frac{\alpha_k}{2\lambda_k} ( \|\theta_{k-1}-\theta\|^2-\|\theta_{k}-\theta\|^2 )+\frac{l\alpha_{k}}{2} \|\theta_{k}^{md}-\theta\|^2+\frac{l\alpha_{k}^2 (1-\alpha)_k}{2} \|\theta_{k-1}^{ag}-\theta_{k-1}\|^2+\alpha_k\eta_k+\chi_k,
	\end{split}
	\end{equation}
	where the second inequality holds from the fact that $ab\leq \frac{a^2+b^2}{2}$.
	
	Again we apply Lemma \ref{lemma.Gamma} and have for $\forall \theta \in \Omega$
	
	\begin{equation}
	\begin{split}
	&\frac{F(\theta_N^{ag})-F(\theta)}{\Gamma_N}+\sum_{k=1}^{N} \frac{1-L\beta_k}{4\beta_k\Gamma_k} \|\theta_k^{ag}-\theta_k^{md}\|^2\\
	\leq &\frac{\|\theta_0-\theta\|^2}{2\lambda_1}+\frac{l}{2} \sum_{k=1}^{N} \frac{\alpha_k}{\Gamma_k} \big( \|\theta_k^{md}-\theta\|^2+\alpha_k(1-\alpha_k)\|\theta_{k-1}^{ag}-\theta_{k-1}\|^2 \big)+\sum_{k=1}^{N}\frac{\chi_k}{\Gamma_k}+ \sum_{k=1}^{N} \frac{\alpha_k}{\Gamma_{k}} \eta_k\\
	&+\sum_{k=1}^{N} \frac{\alpha_k}{\eta_k} \langle \bar{\delta
	}_k,\theta-\theta_{k-1} \rangle+\sum_{k=1}^{N} \frac{\beta_k \|\bar{\delta}_k\|^2}{\Gamma_k (1-L\beta_k)}.    
	\end{split}
	\end{equation}
	
	Now choose $\theta=\theta^*$, where $\theta^*$ is the  optimal solution, take expectation over both side with respect to $\delta_{[N]}$ and use the fact that  $\mathbb{E} \langle \bar{\delta}_k,\theta^*-\theta_{k-1}| \bar{\delta}_{[k-1]} \rangle=0 $ and \eqref{equ.average_variance},we have
	
	\begin{equation}
	\begin{split}
	&\frac{\mathbb{E}_{\delta_{[N]}} F(\theta_N^{ag})-F(\theta^*)}{\Gamma_N}+\sum_{k=1}^{N} \frac{1-L\beta_k}{4\beta_k\Gamma_k} \mathbb{E}_{\delta_{[N]}}\|\theta_k^{ag}-\theta_k^{md}\|^2\\
	\leq &\frac{\|\theta_0-\theta^*\|^2}{2\lambda_1}+\frac{L_f}{\Gamma_N} (\|\theta^*\|^2+2M^2)+\sigma^2 \sum_{k=1}^{N}\frac{\beta_k}{\Gamma_k (1-L\beta_k)m_k}+\sum_{k=1}^{N}\frac{\chi_k}{\Gamma_k}+ \sum_{k=1}^{N} \frac{\alpha_k}{\Gamma_{k}} \eta_k.
	\end{split}
	\end{equation}
	
	Using the definition of approximated gradient mapping $ \tilde{g} (\theta_{k}^{md},\bar{G}_k,\beta_k,\chi_k):=\frac{\theta_{k}^{md}-\theta_k^{ag}}{\beta_k}  $, we have
	
	\begin{equation}
	\begin{split}
	&\sum_{k=1}^{N} \frac{(1-L\beta_k)\beta_k}{4\Gamma_k} \mathbb{E}_{\delta_{[N]}}\|\tilde{g}(\theta_k^{md},\bar{G}_k,\beta_k,\xi_k) \|^2\\
	\leq &\frac{\|\theta_0-\theta^*\|^2}{2\lambda_1}+\frac{l}{\Gamma_N} (\|\theta^*\|^2+2M^2)+\sigma^2 \sum_{k=1}^{N}\frac{\beta_k}{\Gamma_k (1-L\beta_k)m_k}+\sum_{k=1}^{N}\frac{\chi_k}{\Gamma_k}+ \sum_{k=1}^{N} \frac{\alpha_k}{\Gamma_{k}} \eta_k.
	\end{split}
	\end{equation}
	
	Recall the setting of $\chi_k$, $\alpha_k$, $\beta_k$ and \eqref{equ.gamma_relation}, and notice the following fact (using Lemma \ref{lemma.bound_gradient_mapping}) 
	
	$$ \|\tilde{g} (\theta_{k}^{md},\bar{G}_k,\beta_k,\chi_k)- g(\theta_{k}^{md},\bar{G}_k,\beta_k )\|^2\leq \chi_k/\beta_k\leq \frac{L}{N}, $$
	 the fact
	$$ \mathbb{E} \| g(\theta_{k}^{md},\bar{G}_k,\beta_k )- g(\theta_{k}^{md},\nabla F(\theta_k^{md}),\beta_k )\|^2\leq \frac{\sigma^2}{m_k},$$
	
	and if we choose $p_k=\frac{\Gamma_k^{-1} \beta_k (1-L\beta_k)}{\sum_{k=1}^{N}\Gamma_k^{-1}\beta_k (1-L\beta_k)}=\frac{\Gamma_k^{-1}}{\sum_{k=1}^{N}\Gamma_k^{-1}}~,$  (note $\beta_k=\frac{1}{2L}, \Gamma_k=\frac{2}{k(k+1)}  $ by \eqref{equ.gamma_relation}),  
	we have 
	$$ \mathbb{E}  [\|g(\theta_{R}^{md},\bar{G}_R,\beta_R)\|^2]\leq 192L \big(  \frac{4L \|\theta_0-\theta^*\|^2}{N^2(N+1)}+\frac{l}{N} (\|\theta^*\|^2+2M^2)+\frac{3\sigma^2}{LN^3}\sum_{k=1}^{N} \frac{k^2}{m_k} +\frac{1}{N}  \big). $$
	
	Now choose $m_k=k$, we obtain
	
	$$ \mathbb{E}  [\|g(\theta_{k}^{md},\bar{G}_k,\beta_k)\|^2)]\leq 192L \big(  \frac{4L \|\theta_0-\theta^*\|^2}{N^2(N+1)}+\frac{l}{N} (\|\theta^*\|^2+2M^2)+\frac{3\sigma^2}{2LN} +\frac{1}{N}  \big). $$
	
\end{proof}

\begin{proof}[Proof of corollary \ref{Corollary.stochastic}]
	Recall $\bar{G}_k=\frac{1}{m_k}\sum_{i=1}^{m_k} G(x_k^{md},\xi_{k,i})$ , $m_k=k$ using the result of theorem \ref{Theorem.stochastic}, it is easy to see the SFO complexity is $\mathcal{O} (1/\epsilon^2)$.
	
	The proof on the LO complexity is same with option II in theorem \ref{Theorem.batch}. We need to calculate the steps to converges up to the tolerance in the procedure condg.  In particular, we have two calls of condg, where they need $\frac{1}{2\lambda_k}/\eta_k= \lceil \frac{2NL}{k} \rceil$ and  $\frac{1}{2\beta_k}/\chi_k=NL$ steps to converges with tolerance $\eta_k$ and $\chi_k$. Thus the total number of LO is $\mathcal{O} (N^2).$ 
\end{proof}

\begin{proof}
	To achieve $ \mathbb{E}[\| g(\theta_{\alpha},\nabla F(\theta_{\alpha}),\lambda)\|^2] \leqslant \frac{1}{\epsilon}$, the number of total iteration should be $T=\mathcal{O}(\frac{1}{\epsilon})$.
	The LO complexity per each inner iteration is $\mathcal{O}(T)$ by our choice of $\lambda$ and $\eta$, which gives us the overall LO complexity $\mathcal{O}(T^2)=\mathcal{O}(\frac{1}{\epsilon^2})$.
	The IFO complexity is given by $\frac{T}{m}(n+bm)$, sustituting our choice of $m,b$ gives IFO complexity being $\mathcal{O}(\frac{n^{\frac{2}{3}}}{\epsilon})$.
\end{proof}

\subsection{Proof of stochastic finite sum case}
The following lemma is used to control variance of stochastic gradient $v_t^{s+1}$ in non-convex setting.
\begin{lemma}
	In Algorithm \ref{alg:non-convex-finite-sum-SCGS}, we have:
	$$\mathbb{E} [\| \nabla F(\theta_t^{s+1})-v_t^{s+1}\|^2] \leqslant \frac{L^2}{b} \|\theta_t^{s+1}-\tilde{\theta}^s \|^2 $$
\end{lemma}
\begin{proof}
	\begin{align*}
	\mathbb{E} [\| \nabla F(\theta_t^{s+1})-v_t^{s+1} \|^2] & = \mathbb{E} [\| \frac{1}{b} \sum_{i \in I_t}(\nabla f_{i_t}(\theta_t^{s+1})-\nabla f_{i_t}(\tilde{\theta}^s))-(\nabla F(x_t^{s+1})-g^{s+1}) \|^2] \\
	& \leqslant \mathbb{E} [\| \frac{1}{b} \sum_{i \in I_t}(\nabla f_{i_t}(\theta_t^{s+1})-\nabla f_{i_t}(\tilde{\theta}^s)) \|^2]\\
	& \leqslant \mathbb{E} [\frac{1}{b} \sum_{i \in I_t} \| \nabla f_{i_t}(\theta_t^{s+1})-\nabla f_{i_t}(\tilde{\theta}^s) \|^2 ]\\
	& \leqslant \mathbb{E} [\frac{L^2}{b} \sum_{i \in I_t} \| \theta_t^{s+1}-\tilde{\theta}^s \|^2]\\
	& = \frac{L^2}{b} \sum_{i \in I_t} \| \theta_t^{s+1}-\tilde{\theta}^s \|^2
	\end{align*}
	where the first inequality uses bouding variance of random variable by second moment and the second inequality uses the fact that $\mathbb{E} [\|\sum_{i=1}^b X_i \|^2]\leqslant b\mathbb{E}[\sum_{i=1}^b \| X_i \|^2]$.
\end{proof}

We also need the following key lemma.
\begin{lemma}\label{mirror_descent_lemma}
	Let $y=cndg(\omega,x,\lambda,\eta)$, then we have:
	\begin{equation}\label{mirror_descent}
	F(y)\leqslant F(z)+\inner{y-z}{\nabla F(x)-\omega}+(\frac{L}{2}-\frac{1}{2\lambda})\|y-x\|^2+(\frac{L}{2}+\frac{1}{2\lambda})\|z-x\|^2-\frac{1}{2\lambda}\|y-z\|^2+\eta,\quad \forall z \in \mathbb{R}^d.
	\end{equation}
\end{lemma}
\begin{proof}
	By termination criteria of $cndg$ procedure, we have
	\begin{equation}
	\inner{\omega+\frac{1}{\lambda}(y-x)}{y-z}\leqslant \eta
	\end{equation}
	re-arrange terms we have
	\begin{align}
	\inner{\omega}{y-z} & \leqslant \frac{1}{\lambda}\inner{x-y}{y-z}+\eta \nonumber\\
	& = \frac{1}{2\lambda}(\|x-z\|^2-\|y-z\|^2-\|x-y\|^2)+\eta \label{proxineq1}.
	\end{align}
	Now by smoothness of $F(x)$ we have:
	\begin{align}
	F(y) & \leqslant F(x)+\inner{\nabla F(x)}{y-x}+\frac{L}{2}\|y-x\|^2 \nonumber \\
	& \leqslant F(z)+\inner{\nabla F(x)}{x-z}+\frac{L}{2}\|z-x\|^2+\inner{\nabla F(x)}{y-x}+\frac{L}{2}\|y-x\|^2 \nonumber\\
	& = F(z)+\inner{\nabla F(x)}{y-z}+\frac{L}{2}\|z-x\|^2+\frac{L}{2}\|y-x\|^2 \label{proxineq2}.
	\end{align}
	Add (\ref{proxineq1}) and (\ref{proxineq2}) together the result follows.
\end{proof}

Now we are ready to prove Theorem \ref{finite_sum_NCGS}.
\begin{proof}
	We first define $\hat{\theta}_{t+1}^{s+1}=\varphi(\theta_t^{s+1},\nabla F(\theta_t^{s+1}), \lambda)$. Then by a direct application of Lemma \ref{mirror_descent_lemma}(with $y=\hat{\theta}_{t+1}^{s+1}, z=x=\theta_t^{s+1}$), we have
	\begin{equation}
	F(\hat{\theta}_{t+1}^{s+1}) \leqslant F(\theta_t^{s+1})+(\frac{L}{2}-\frac{1}{2\lambda})\|\hat{\theta}_{t+1}^{s+1}-\theta_t^{s+1}\|^2-\frac{1}{2\lambda}\|\hat{\theta}_{t+1}^{s+1}-\theta_t^{s+1}\|^2 \label{app_ineq1}
	\end{equation}
	By a second application of Lemma \ref{mirror_descent_lemma}(with $y=\theta_{t+1}^{s+1}, z=\hat{\theta}_{t+1}^{s+1}, x=\theta_t^{s+1}$), we have
	\begin{align}
	F(\theta_{t+1}^{s+1}) & \leqslant F(\hat{\theta}_{t+1}^{s+1})+\inner{\theta_{t+1}^{s+1}-\hat{\theta}_{t+1}^{s+1}}{\nabla F(\theta_t^{s+1})-v_t^{s+1}}+(\frac{L}{2}-\frac{1}{2\lambda})\|\theta_{t+1}^{s+1}-\theta_{t}^{s+1}\|^2 \nonumber \\
	& \quad + (\frac{L}{2}+\frac{1}{2\lambda})\|\hat{\theta}_{t+1}^{s+1}-\theta_{t}^{s+1}\|^2-\frac{1}{2\lambda}\|\theta_{t+1}^{s+1}-\hat{\theta}_{t+1}^{s+1}\|^2+\eta  \nonumber \\
	& \leqslant F(\hat{\theta}_{t+1}^{s+1}) + \frac{1}{2\lambda}\|\theta_{t+1}^{s+1}-\hat{\theta}_{t+1}^{s+1}\|^2+\frac{\lambda}{2}\|\nabla F(\theta_t^{s+1})-v_t^{s+1} \|^2 + (\frac{L}{2}-\frac{1}{2\lambda})\|\theta_{t+1}^{s+1}-\theta_{t}^{s+1}\|^2 \nonumber \\
	& \quad + (\frac{L}{2}+\frac{1}{2\lambda})\|\hat{\theta}_{t+1}^{s+1}-\theta_{t}^{s+1}\|^2-\frac{1}{2\lambda}\|\theta_{t+1}^{s+1}-\hat{\theta}_{t+1}^{s+1}\|^2+\eta  \nonumber \\
	& \leqslant F(\hat{\theta}_{t+1}^{s+1}) +\frac{\lambda L^2}{2b}\|\theta_t^{s+1}-\tilde{\theta}^s\|^2 \nonumber \\
	& \quad +(\frac{L}{2}-\frac{1}{2\lambda})\|\theta_{t+1}^{s+1}-\theta_{t}^{s+1}\|^2 + (\frac{L}{2}+\frac{1}{2\lambda})\|\hat{\theta}_{t+1}^{s+1}-\theta_{t}^{s+1}\|^2+\eta \label{app_ineq2}
	\end{align}
	Then by adding (\ref{app_ineq1}) and (\ref{app_ineq2}) together we have
	\begin{align}
	F(\theta_{t+1}^{s+1}) & \leqslant F(\theta_t^{s+1})+(L-\frac{1}{2\lambda})\|\hat{\theta}_{t+1}^{s+1}-\theta_t^{s+1}\|^2 + (\frac{L}{2}-\frac{1}{2\lambda})\|\theta_{t+1}^{s+1}-\theta_{t}^{s+1}\|^2 \nonumber \\
	& \quad +\frac{\lambda L^2}{2b}\|\theta_t^{s+1}-\tilde{\theta}^s\|^2+\eta \label{value_ineq}
	\end{align}
	
	Now we define Lyapunov function as follows, with $c_m=0$ and $c_t=c_{t+1}(1+\beta)+\frac{\lambda L^2}{2b}$.
	$$L_t^{s+1}=F(\theta_{t}^{s+1})+c_t\|\theta_{t}^{s+1}-\tilde{\theta}^s\|^2$$
	By (\ref{value_ineq}), we have:
	\begin{align}
	L_{t+1}^{s+1} & = F(\theta_{t+1}^{s+1})+c_{t+1}\|\theta_{t+1}^{s+1}-\tilde{\theta}^s\|^2 \nonumber \\
	& \leqslant F(\theta_t^{s+1})+(L-\frac{1}{2\lambda})\|\hat{\theta}_{t+1}^{s+1}-\theta_t^{s+1}\|^2 + (\frac{L}{2}-\frac{1}{2\lambda})\|\theta_{t+1}^{s+1}-\theta_{t}^{s+1}\|^2 \nonumber \\
	& \quad +\frac{\lambda L^2}{2b}\|\theta_t^{s+1}-\tilde{\theta}^s\|^2+\eta+c_{t+1}\|\theta_{t+1}^{s+1}-\tilde{\theta}^s\|^2 \nonumber \\
	& \leqslant F(\theta_t^{s+1})+(L-\frac{1}{2\lambda})\|\hat{\theta}_{t+1}^{s+1}-\theta_t^{s+1}\|^2 + (\frac{L}{2}-\frac{1}{2\lambda})\|\theta_{t+1}^{s+1}-\theta_{t}^{s+1}\|^2 \nonumber \\
	& \quad +\frac{\lambda L^2}{2b}\|\theta_t^{s+1}-\tilde{\theta}^s\|^2+\eta+c_{t+1}(1+\frac{1}{\beta})\|\theta_{t+1}^{s+1}-\theta_{t}^{s+1}\|^2+c_{t+1}(1+\beta)\|\theta_{t}^{s+1}-\tilde{\theta}^{s}\|^2 \nonumber \\
	& = F(\theta_t^{s+1})+(L-\frac{1}{2\lambda})\|\hat{\theta}_{t+1}^{s+1}-\theta_t^{s+1}\|^2 + [c_{t+1}(1+\frac{1}{\beta})+\frac{L}{2}-\frac{1}{2\lambda}]\|\theta_{t+1}^{s+1}-\theta_{t}^{s+1}\|^2 \nonumber \\
	& \quad +[c_{t+1}(1+\beta)+\frac{\lambda L^2}{2b}]\|\theta_t^{s+1}-\tilde{\theta}^s\|^2+\eta \nonumber \\
	& = F(\theta_t^{s+1})+(L-\frac{1}{2\lambda})\|\hat{\theta}_{t+1}^{s+1}-\theta_t^{s+1}\|^2 + [c_{t+1}(1+\frac{1}{\beta})+\frac{L}{2}-\frac{1}{2\lambda}]\|\theta_{t+1}^{s+1}-\theta_{t}^{s+1}\|^2 \nonumber \\
	& \quad +[c_{t+1}(1+\beta)+\frac{\lambda L^2}{2b}]\|\theta_t^{s+1}-\tilde{\theta}^s\|^2+\eta \nonumber \\
	& \leqslant F(\theta_t^{s+1})+(L-\frac{1}{2\lambda})\|\hat{\theta}_{t+1}^{s+1}-\theta_t^{s+1}\|^2 +c_t\|\theta_t^{s+1}-\tilde{\theta}^s\|^2+\eta \nonumber \\
	& = L_t^{s+1}+(L-\frac{1}{2\lambda})\|\hat{\theta}_{t+1}^{s+1}-\theta_t^{s+1}\|^2 +\eta \label{telescope_ineq}
	\end{align}
	where the first inequality comes from (\ref{value_ineq}), the second inequality comes from Cauchy-Schwarz inequality and the final inequality comes from definition of $c_t$ and the fact that
	$c_{t+1}(1+\frac{1}{\beta})+\frac{L}{2}-\frac{1}{2\lambda}\leqslant 0$ for appropriate choice of $\beta$ and $\lambda$ which we now verify.
	By definition of $c_t$ we can easily find
	\begin{align}
	c_t & =\frac{\lambda L^2}{2b}\frac{(1+\beta)^{m-t}-1}{\beta} \nonumber \\
	& = \frac{\lambda L^2 m}{2b}((1+\frac{1}{m})^{m-t}-1) \quad (let \beta=\frac{1}{m}) \nonumber \\
	& \leqslant \frac{\lambda L^2 m}{2b}(e-1)  \leqslant \frac{\lambda L^2 m}{b}
	\end{align}
	Hence we have
	\begin{align}
	c_{t+1}(1+\frac{1}{\beta})+\frac{L}{2}-\frac{1}{2\lambda} & \leqslant \frac{\lambda L^2 m}{b}(1+m)+\frac{L}{2}-\frac{1}{2\lambda} \nonumber \\
	& \leqslant \frac{2 \lambda L^2 m^2}{b}+\frac{L}{2}-\frac{1}{2\lambda} \nonumber \\
	& \leqslant \frac{1}{2\lambda}(\frac{4\lambda^2 L^2 m^2}{b}+L\lambda-1) \leqslant 0 \label{par_condition}
	\end{align}
	where the last inequality comes from plug in back our specification of $\lambda,b,m$ in our theorem.
	Now by telescoping both side of (\ref{telescope_ineq}), we have:
	\begin{align}
	L_m^{s+1}+(\frac{1}{2\lambda}-L)\sum_{t=0}^{m-1} \|\hat{\theta}_{t+1}^{s+1}-\theta_t^{s+1}\|^2 \leqslant L_0^{s+1}+m\eta  \label{epoch_ineq1}
	\end{align}
	By using $c_m=0$ and that $\tilde{\theta}^{s+1}=\theta_m^{s+1}$ we have $L_m^{s+1}=F(\theta_m^{s+1})=F(\tilde{\theta}^{s+1})$, by using $\theta_0^{s+1}=\tilde{\theta}^s$, we have
	$L_0^{s+1}=F(\theta_0^{s+1})=F(\tilde{\theta}^s)$. Hence (\ref{epoch_ineq1}) becomes
	\begin{equation}
	(\frac{1}{2\lambda}-L)\sum_{t=0}^{m-1} \|\hat{\theta}_{t+1}^{s+1}-\theta_t^{s+1}\|^2 \leqslant F(\tilde{\theta}^s)-F(\tilde{\theta}^{s+1})+m\eta  \label{epoch_ineq2}
	\end{equation}
	Now telescope through all the epoch, we have:
	\begin{align}
	(\frac{1}{2\lambda}-L)\sum_{t=0}^{m-1} \sum_{s=0}^{S} \|\hat{\theta}_{t+1}^{s+1}-\theta_t^{s+1}\|^2 & \leqslant F(\tilde{\theta}^0)-F(\tilde{\theta}^{S+1})+T\eta \nonumber \\
	& \leqslant F(\theta_0)-F(\theta^{\star})+1
	\end{align}
	Now by definition of gradient mapping we have $\|\hat{\theta}_{t+1}^{s+1}-\theta_t^{s+1}\|^2=\lambda^2 \|g(\theta_t^{s+1},\nabla F(\theta_t^{s+1}), \lambda)\|^2$. Thus by definition of $\theta_{\alpha}$ we have:
	\begin{equation}
	\lambda^2(\frac{1}{2\lambda}-L)\mathbb{E} \|g(\theta_{\alpha},\nabla F(\theta_{\alpha}), \lambda)\|^2 \leqslant \frac{F(\theta_0)-F(\theta^{\star})+1}{T}
	\end{equation}
	plug in back the choice of $\lambda=\frac{1}{3L}$, the claim follows immediately.
\end{proof}

\end{document}